\documentclass[10pt,twoside,final,psamsfonts]{amsart}
\usepackage[draft=false]{hyperref}
\usepackage[psamsfonts]{amssymb}
\usepackage{times,a4wide}

\theoremstyle{plain}
\newtheorem*{theorem*}{Theorem}

\newtheorem{theorem}{Theorem}[section]
\newtheorem{proposition}[theorem]{Proposition}
\newtheorem*{proposition*}{Proposition}

\newtheorem*{corollary*}{Corollary}
\newtheorem{lemma}[theorem]{Lemma}
\newtheorem{definition}[theorem]{Definition}
\newtheorem*{lemma*}{Lemma}
\newtheorem*{definition*}{Definition}

\newtheorem{remark}[theorem]{Remark}
\newtheorem*{remark*}{Remark}
\newtheorem*{conjecture*}{Conjecture}
\theoremstyle{definition}

\newcommand{\C}{\mathbb{C}}

\newcommand{\N}{\mathbb{N}}
\newcommand{\R}{\mathbb{R}}

\renewcommand{\Re}{\operatorname{Re}}

\title[Interpolation by  entire functions with growth conditions]
{Interpolation by entire functions with growth conditions}

\author[ M. Ouna\"{\i}es]
{ Myriam Ounaies}

\address{Institut de Recherche Math\'ematique Avanc\'ee, Universit\'e 
Louis Pasteur 7 Rue Ren\'e Des\-car\-tes, 67084 Strasbourg CEDEX, France.}
\email{ounaies@math.u-strasbg.fr}

\date{\today}



\begin{document}

\maketitle
\section*{Introduction}

Let $p : \C \rightarrow [0,+\infty[$ be a weight (see Definition \ref{defweight}) and $A_p(\C)$ the vector space of all entire functions satisfying $\sup_{z\in \C} \vert f(z)\vert \le \exp(-B p(z))<\infty$ for some constant $B>0$.  For instance, if $p(z)=\vert z\vert$, $A_p(\C)$ is the space of all entire functions of exponential type.

Following \cite{Be-Ta}, the interpolation problem  we are considering is  :  let $V=\{(z_j,m_j)\}_j$ be a multiplicity variety, that is, $\{z_j\}_j$ is a sequence of complex numbers diverging to $\infty$,  $\vert z_j\vert \le \vert z_{j+1}\vert $ and $\{m_j\}_j$ is a sequence of strictly positive integers.  Let  $\{w_{j,l}\}_{j,0\le l<m_j}$ be a  doubly indexed sequence of complex numbers. 

Under what conditions does there exist an entire function $f\in A_p(\C)$ such that
 $$\frac{f^{(l)}(z_j)}{l!}=w_{j,l}, \ \ \forall j,\ \ \forall 0\le l<m_j ?$$ 

In other words, if we denote by $\rho$  the restriction operator defined on $A_p(\C)$ by 
\[ 
\rho(f)=\{\frac{f^{l}(z_j)}{l!}\}_{j,0\le l< m_j},
\]
what is the image of $A_p(\C)$ by $\rho$ ?

  We say that $V$ is an "interpolating variety" when  $\rho(A_p(\C))$ is the space of all  doubly indexed sequence $W=\{w_{j,l}\}$ satisfying the  growth condition 
 $$\vert w_{j,l}\vert \le A \exp(Bp(z_j))\ \ \forall j, \  \ \forall 0\le l<m_j,$$
  for certain constants $A,B>0$.  
  
Let us mention the important following result :

\begin{theorem}\label{intro}\cite[Corollary 4.8]{Be-Li}  

$V$ is an interpolating variety for $A_p(\C)$ if and only if the following conditions hold :
\begin{itemize}
\item[(i)] $\ \forall R>0,\ \ N(0,R)\le A p(R)+B$
\item[(ii)] $\forall j\in \N,\ \ \ N(z_j, \vert z_j\vert) \le A\ p(z_j)+B$,
\end{itemize}
for some constants $ A, B>0$.  
\end{theorem}

Here, $N(z,r)$ denotes the integrated counting function of $V$ in the disc of center $z$ and radius $r$ (see Definition \ref{count} below).

In \cite{Be-Ta}, Berenstein and Taylor describe  the space $\rho(A_p(\C))$ in the case where there exists a function $g\in A_p(\C)$ such that $V= g^{-1}(0)$. They used groupings of the points of $V$ with respect to the connex components of the set $\{\vert g(z)\vert \le \varepsilon \exp(-Bp(z))\}$, for some $\varepsilon, B>0$ and the divided diffrences with respect to this grouping.

 The main aim of this paper is to determine more explicitely the space $\rho(A_p(\C))$ in the more general case  where condition (i) is satisfied. It is clear that it is the case when $V$ is not a uniqueness set for $A_p(\C)$, that is, when there exists  $f\in A_p(\C)$ not identically equal to zero such that  $V\subset f^{(-1)}(0)$.  
 
 We refer to \cite{Bo-La}  and \cite{Le} for similar results in the case where $p(z)=\vert z\vert^{\alpha}$.  
 
  As in \cite{Be-Ta} and \cite{Bo-La},  the divided differences will be important tools. Our condition will involve the divided differences with respect to the intersections of $V$ with discs centered at the origin. To be more precise,  the main theorem, stated  in the case where all the multiplicities are equal to one, for the sake of simplicity, is the following :

\begin{theorem}
Assume that $V$ verifies condition (i). Then $W=\{w_j\}_j \in \rho(A_p(\C))$ if and only if for all $R>0$,
\[
\vert  \sum_{\vert z_k\vert<R} w_k \prod_{\vert z_m\vert<R, m\not=k} R/(z_k-z_m)    \vert \le A \exp{B p(R)},
\]

\end{theorem}
where $A,B>0$ are positive constants only depending on $V$ and $W$.

  We will denote by $\tilde A_p(V)$ the space of sequences $W=\{w_j\}_j$ satisfying the above condition. We will show that in general $\rho(A_p(\C))\subset \tilde A_p(V)$, thus, we can consider $\rho : A_p(\C)\rightarrow \tilde A_p(V)$.
In this context,  the theorem states that condition (i) implies the surjectivity of $\rho$.
 
On the other hand,  we will prove that condition (i) is actually equivalent to saying that $V$ is not a uniqueness set or, in other words, it is equivalent to the non-injectivity of $\rho$.

  As a corollary of the main theorem, we will  find the sufficency in the geometric characterization of interpolating varieties given  in Theorem \ref{intro}.

The difficult part  of the proof of the main theorem is the sufficiency. As in \cite{Be-Or, Ha-Ma, Ma-Or-Ou}, we will follow a Bombieri-H\"ormander approach based on $L^2$-estimates on the solution to the $\bar\partial$-equation. The scheme will be the following : the condition on $W$  gives a smooth interpolating function $F$ with a good growth, using a partition of the unity and Newton polynomials (see Lemma \ref{F}).
 Then we are led to solve the $\bar\partial$ equation : $\bar\partial u=-\bar\partial F$ with $L^2$-estimates, using H\"ormander theorem \cite{Ho1}. To do so, we need to construct a subharmonic function $U$ with a convenient growth and with prescribed singularities on the points $z_j$ (see Lemma \ref{U}). Following Bombieri \cite{Bo}, the fact that $e^{-U}$ is not summable near the points  $\{z_j\}$ forces  $u$ to vanish on the points $z_j$ and we are done by defining the interpolating entire function by $u+F$.

A final remark about the notations :

$A$, $B$ and $C$ will denote positive constants and 
their actual value may change from one  occurrence to the next. 

$A(t)\lesssim B(t)$ means that there exists a constant $C>0$, not depending on $t$ such that $A(t)\le C B(t)$. 
$A\simeq B$ means that $A\lesssim B\lesssim A$.

The notation $D(z,r)$ will be used for  the euclidean disk of center $z$ and radius $r$. We will denote $\displaystyle \partial f=\frac{\partial f}{\partial z}$, $\displaystyle \bar \partial f=\frac{\partial f}{\partial \bar z}$. Then  $\Delta f=4\partial\bar\partial f$ denotes the laplacian of $f$.

\section{Preliminaries and definitions.}

\begin{definition}\label{defweight}

A subharmonic function $p:\C\longrightarrow\R_+$,  is called a weight 
if, for some positive constants $C$,
\begin{itemize}
\item[(a)] $\ln(1+|z|^2)\le C p(z)$,
\item[(b)]  $p(z)=p(\vert z\vert)$,
\item[(c)] there exists a constant $C>0$  such that $p(2z)\le C p(z)$.
\end{itemize}
\end{definition}

Property (c) is referred to as the "doubling property of the weight $p$". It implies that $p(z)=O(\vert z\vert^\alpha)$ for some $\alpha>0$.

Let $A(\C)$ be the set of all entire functions, we consider the space
$$A_p(\C)=\Bigl\{f\in A(\C),\ \  \forall z\in \C,\  |f(z)| \le A\, e^{B p(z)}\hbox{ for some } A>0, B>0\Bigr\}.$$

\begin{remark}\label{rem}{}\par

\begin{itemize}
\item[(i)] Condition (a) implies that $A_p(\C)$ contains all polynomials.
\item[(ii)] Condition (c) implies that $A_p(\C)$ is stable under differentiation.
\end{itemize}
\end{remark}

Examples :
\begin{itemize}
\item $p(z)=\ln(1+\vert z\vert^2)$. Then $A_p(\C)$ is the space of all the polynomials.
\item $p(z)=\vert z\vert$. Then $A_p(\C)$ is  the space of entire functions of exponential type.
\item $p(z)=\vert z\vert^{\alpha}$, $\alpha>0$. Then $A_p(\C)$ is the space of all entire functions of order $\le \alpha$ and finite type.
\end{itemize}

Let  $V=\{(z_j,Êm_j)\}_{j\in \N}$ be a multiplicity variety.

For a function $f\in A(\C)$, we will  write  $V=f^{-1}(0)$ when $f$ vanishes exactly on the points $z_j$ with multiplicity $m_j$ and $V\subset f^{-1}(0)$ when $f$ vanishes on the points $z_j$ (but possibly elsewhere) with multiplicity at least equal to $m_j$.

We will say that $V$ is a uniqueness set  for $A_p(\C)$ if there is no function $f\in A_p(\C)$, except the zero function, such that $V\subset f^{-1}(0)$.

We need to recall the definitions of the counting functions and the integrated counting functions :

\begin{definition}\label{count}
Let $V=\{(z_j,m_j)\}_j$ be a multiplicity variety. For $z\in \C$ and $r>0$,
\[
n(z,r)=\sum\limits_{|z-z_j|\le r} m_j,
\]

\[
N(z,r)=\int_0^r\frac{n(z,t)-n(z,0)}t\, dt + n(z,0)\ln r=\sum_{0<\vert z- z_j\vert\le r}m_j \ln \frac{r}{\vert z-z_j\vert}+n(z,0)\ln r ,
\]
\end{definition}

An application of Jensen's formula in the disc $D(0,R)$ shows that, if $V$ is not a uniqueness set for $A_p(\C)$, then the following condition holds :

\begin{equation}\label{N(0,R)}
\exists A,B>0, \ \forall R>0,\ \ N(0,R)\le A p(R)+B
\end{equation}

We will lately show that the converse property holds.

By analogy with the spaces $A(\C)$ and $A_p(\C)$, we define the following spaces

$$A(V)=\{W=\{w_{j,l}\}_{j,0\le l < m_j}\subset \C\}$$
and
$$A_p(V)=\Bigl\{W=\{w_{j,l}\}_{j,0\le l < m_j}\subset \C, \ \  \forall j,\ \sum_{l=0}^{m_j-1}|w_{j,l}| \le A\, e^{B p(z_j )} \hbox{ for some } A>0, B>0\Bigr\}.$$

The space $A_p(\C)$ can be seen as the union of the Banach spaces 
$$A_{p,B}(\C)=\{f\in A(\C), \ \ \Vert f\Vert_B:=\sup_{z\in \C}\vert f(z)\vert e^{-Bp(z)} <\infty\}$$ 
and has a structure of an (LF)-space
with the topology of the inductive limit.
The analog is true about $A_p(V)$.

\begin{remark}(see \cite[Proposition 2.2.2]{Be-Ga})\par
Let $f$ be a function in $A_p(\C)$. Then, 
for some constants $A>0$ and $B>0$,  
 $$\forall z\in \C,\ \ \sum_{k=0}^{\infty} \left\vert \frac{f^{(k)}(z)}{k!}\right\vert\le A e^{Bp(z)}.$$ 
  \end{remark}

As a consequence of this remark, we see that the restriction map :
\[ 
\begin{split}
\rho: A & (\C)\longrightarrow A (V)\\
& f\quad \mapsto\; \{\frac{f^{l}(z_j)}{ l!}\}_{j,0\le l\le m_j-1}
\end{split} 
\]
maps $A_p(\C)$ into $A_p(V)$, but in general, the space $A_p(V)$ is  larger than $\rho(A_p(\C))$. It is clear that $\rho$ is injective if and only if $V$ is a uniqueness set for $A_p(\C)$.

When $\rho(A_p(\C))$=$A_p(V)$, we say that $V$ is an interpolating variety for $A_p(\C)$. 
As mentioned in the introduction, Berenstein and Li gave a geometric characterization of these varieties :

\begin{theorem}\label{doub}\cite[Corollary 4.8]{Be-Li}  

$V$ is an interpolating variety for $A_p(\C)$ if and only if conditions (\ref{N(0,R)}) and 

\begin{equation}\label{N(z,z)}
\exists A>0,\  \exists B>0\ \ \forall j\in \N,\ \ \ N(z_j,\vert z_j\vert) \le A\ p(z_j)+B
\end{equation}
hold.

\end{theorem}

In this paper, we are concerned by determining the subpace   $\rho(A_p(\C))$ of $A(V)$ in the case where condition (\ref{N(0,R)}) is verified.

To any $W=\{w_{j,l}\}_{j,0\le l\le m_j-1}\in A(V)$, we associate  the sequence of divided differences $\Phi(W)=\{\phi_{j,l}\}_{j,0\le l\le m_j-1}$ defined by induction as follows :

We will denote by
$$\Pi_q(z)=\prod_{k=1}^q (z-z_k)^{m_k}, \hbox{for all } q\ge 1.$$

$$\phi_{1,l}=w_{1,l}, \hbox{for all}\ \  0\le l\le m_1-1, $$
$$\phi_{q,0}=\frac{w_{q,0}-P_{q-1}(z_q)}{\Pi_{q-1}(z_q)},$$
$$ \phi_{q,l}=\frac{w_{q,l}-\frac{P_{q-1}^{(l)}(z_q)}{l!}-\sum_{j=0}^{l-1}\frac{1}{(l-j)!}\Pi_{q-1}^{(l-j)}(z_q)\phi_{q,j}}{\Pi_{q-1}(z_q)}\ \ \hbox{for } 1\le l\le m_q-1$$
where
$$P_{q-1}(z)=\sum_{j=1}^{q-1} \left(\sum_{l=0}^{m_j-1} \phi_{j,l} (z-z_j)^l\prod_{t=1}^{ j-1} (z-z_t)^{m_t}\right).$$

\begin{remark}\label{Newton}

 Actually,  $P_q$ is the polynomial interpolating the values $w_{j,l}$ at the points $z_j$ with multiplicity $m_j$, for $1\le j\le q$. It is the unique polynomial of degree $m_1+\cdots+ m_q-1$ such that 
$$\frac{P_q^{(l)}(z_j)}{l!}=w_{j,l}$$
for all $1\le j\le q$ and $0\le l \le m_j-1$.
\end{remark}

Examples. 

\begin{itemize}
\item Let $W_0=\{\delta_{1,j}\delta_{l,m_1-1}\}_{j,0\le l<m_j}$. 

Using the fact that $P_j(z)$ must coincide with $\displaystyle (z-z_1)^{m_1-1}\prod_{k=2}^{j-1}(z-z_j)^{m_j}$ and identifying the coefficient in front of $z^{m_1+\cdots+m_{j-1}+l-1}$, we find : 
\[ 
\phi_{1,1}=\phi_{1,2}=\cdots=\phi_{1,m_1-2}=0,\ \ \phi_{1,m_1-1}=1,
\]
and, for $j\ge 2$, $0\le l\le m_j-1$,
\[
\phi_{l,j}=(z_1-z_j)^{-(l+1)}\prod_{k=2}^{j-1} (z_1-z_k)^{-m_k}.
\]

\item In the special case where $m_j=1$ for all $j$ and  $W=\{w_j\}_j$, we have for all $j\ge 1$,
$$\phi_j=\sum_{k=1}^j w_k \prod_{1\le l\le j, l\not=k}(z_k-z_l)^{-1}.$$
To compute the coefficients, we may use the fact that $P_j(z)$ must coincide with  the Lagrange polynomial $\displaystyle \sum_{n=1}^j w_n\prod_{1\le k\le j, k\not=n} \frac{(z-z_k)}{(z_n-z_k)}$ and identify the coefficient in front of $z^{j-1}$.
\end{itemize}
\bigskip

Let us denote by $\tilde A_p(V)$ the subspace of $A(V)$ consisting of the elements $W\in A(V)$ such that the following condition holds :

\begin{equation}\label{Im}
 \textrm{ for all }  n\ge 0,\, \vert z_j\vert\le 2^n\, \textrm{ and }  0\le l\le m_j-1,\ \  \vert \phi_{j,l}\vert 2^{n(l+m_1+...+m_{j-1})}\le A \exp(B p(2^n)),
\end{equation}
where $A$ and $B$ are positive constants only depending on $V$ and $W$. 

We have chosen to use a covering of the complex plane by discs $D(0,2^n)$, but we can replace $2^n$ by any $R^n$ with $R>1$.

\begin{lemma}\label{W0}
Assume $z_1=0$. Then,  condition (\ref{N(0,R)}) holds if and only if
 $$W_0=\{\delta_{1,j}\delta_{l,m_1-1}\}_{j,0\le l<m_j} \in \tilde A_p(V).$$
\end{lemma}

\begin{proof}
Suppose that (\ref{N(0,R)}) is verified. Let $n\in \N$, $0<\vert z_j\vert \le 2^n$ and $0\le l\le m_j-1$.
We have by definition, 
\[ N(0,2^n)=\sum_{0<\vert z_k\vert\le 2^n}m_k\ln\frac{2^n}{\vert z_k\vert}+m_1\ln (2^n)\ge \ln \left( 2^{n(m_1+\cdots+m_j)}\prod_{k=2}^j \vert z_k\vert^{-m_k}\right),
\]
\[
\vert \phi_{j,l}\vert= \vert z_j\vert^{m_j-l-1}\prod_{k=2}^j \vert z_k\vert ^{-m_k}\le 2^{n(m_j-l-1)} \prod_{k=2}^j \vert z_k\vert ^{-m_k} \le \exp(N(0,2^n))2^{-n(m_1+\cdots+m_{j-1}+l+1)}.
\]
We readily obtain the estimate (\ref{Im}), using that $N(0,2^n)\le A p(2^n)+B$.

 Conversely, let $n$ be an integer. Using the estimate (\ref{Im}) when $j\ge 2$ is the number of distinct points $\{z_k\}$ in $D(0,2^n)$ and $l=m_j-1$, we have
 \[
 N(0,2^n)=\ln \left( 2^{n(m_1+\cdots+m_j)}\prod_{k=2}^j \vert z_k\vert^{-m_k}\right)=\ln (2^{n(m_1+\cdots+m_j)}\vert \phi_{j,m_j-1}\vert )\le A p(2^n)+B.
 \]
Then, we deduce the estimate for $N(0,R)$ using the above one with $2^{n-1}\le R<2^n$ and the doubling property of $p$. 
\end{proof}

We define the following norm :

\[
\Vert W \Vert_B=\sup_n \Vert W^{(n)}\Vert_n \exp{(-B p(2^n))}
\]
where 
\[ 
\Vert W^{(n)}\Vert_n=\sup_{\vert z_j\vert\le 2^n}\sup_{0\le l\le m_j-1}\vert \phi_{j,l}\vert 2^{-n(l+m_1+...+m_{j-1})},
\]

The space $\tilde A_p(V)$ can also be seen as an (LF)-space as an inductive limit of the Banach spaces 
$$\tilde A_{p,B}(V)=\{W\in A(V),\ \Vert W\Vert_B <\infty\}.$$

We are now ready to state the main results.

\begin{proposition}\label{nec}
The restriction operator $\rho$ maps continously $A_p(\C)$ into $\tilde A_p(V)$.
\end{proposition}

\begin{proposition}\label{subspace}
Under the assumption of condition (\ref{N(0,R)}),
$\tilde A_p(V)$ is a subspace of $A_p(V)$.
\end{proposition}

 \begin{proposition}\label{intvar}
 If conditions (\ref {N(0,R)}) and (\ref{N(z,z)}) are verified, then $\tilde A_p(V)=A_p(V)$. 
\end{proposition}

\begin{theorem}\label{main}
If condition (\ref{N(0,R)}) holds, then 
$$\tilde A_p(V)=  \rho(A_p(\C)).$$

In other words, condition  (\ref{N(0,R)}) implies that the map $\rho : A_p(\C)\rightarrow \tilde A_p(V)$ is surjective. 
\end{theorem}

The combination of Proposition \ref{intvar} and Theorem \ref{main} shows easily the sufficiency in Theorem \ref{doub}.

Using the  results given so far, we can deduce next theorem  : 

\begin{theorem}\label{unicity}
The following assertions are equivalent :
\begin{item}
\item[(i)] $V$ is not a uniqueness set for $A_p(\C)$.
\item[(ii)] The map $\rho$ is not injective.
\item[(iii)] $V$ verifies condition (\ref{N(0,R)}).
\item[(iv)] The sequence $W_0=\{\delta_{1,j}\delta_{l,m_1-1}\}_{j,0\le l<m_j}$ belongs to $\rho(A_p(\C))$.
\end{item}
\end{theorem}

In particular, it shows that condition (\ref{N(0,R)}) is equivalent to the existence of a function  $f\in A_p(\C)$ such that $V\subset f^{-1}(0)$. 
Combined with Theorem \ref{main}, it shows that, if $\rho$ is not injective, then it is surjective and that, if the image contains $W_0$, then it contains the whole $\tilde A_p(V)$.

\begin{proof}[Proof of Theorem \ref {unicity}]\par
As we mentioned before, it is clear that (i) is equivalent to (ii) and that  (i) implies (iii).

 (iv) implies (i) :
We have a function $f\in A_p(\C)$ not identically equal to $0$ such that $f^{(l)}(z_j)=0$ for all $j\not=1$ and for all $0\le l<m_j$.
The function $g$ defined by $g(z)=(z-z_1)^{m_1} f(z)$ belongs to $A_p(\C)$, thanks to property (i) of the weight $p$, and vanishes on every $z_j$ with multiplicity at least $m_j$. 

 (iii) implies (iv) :

Up to a translation, we may suppose that $z_1=0$. By Lemma \ref{W0}, we know that $W_0\in \tilde A_p(\C)$. By Theorem \ref{main}, $W_0\in \rho(A_p(\C))$.

\end{proof}

\section{Proof of the main results.}

\begin{proof}[Proof of Theorem \ref{nec}]\par

We will first recall some definitions about the divided differences an Newton polynomials. We refer the reader to \cite[Chapter 6.2]{Be-Ga} or \cite[Chapter 6]{Is-Ke} for more details.

Let $f\in A(\C)$ and $x_1,\ldots,x_q$ be distinct points of $\C$. The $q$th divided difference of the function $f$ with respect to the  points $x_1,\ldots,x_q$ is defined by
$$\Delta^{q-1}f(x_1,\ldots,x_q)=\sum_{j=1}^q f(z_j)\prod_{1\le k\le q, k\not=j}(x_j-x_k)^{-1}$$
and the  Newton polynomial of $f$ of degree $q-1$ is
$$P(z)=\sum_{j=1}^q\Delta^{j-1} f(x_1,\ldots,x_j)\prod_{k=0}^{j-1}(z-x_k).$$
It is the unique polynomial of degree $q-1$ such that $P_q(z)=f(x_j)$ for all $1\le j\le q$.

When $x_j$, $1\le j\le q$ are each one repeated $l_j$ times, the divided differences are defined by
\begin{equation*}
\begin{split}
\Delta^{l_1+\cdots+l_q-1} & f(\underbrace{x_1,...,x_1}_{l_1},\dots,\underbrace{x_{q-1},...,x_{q-1}}_{l_{q-1}},\underbrace{x_q,...,x_q}_{l_q})\\
& = \frac{1}{l_1! \cdots l_q!}\frac{\partial^{l_1+\cdots+l_j}}{\partial x_1^{l_1}\cdots \partial x_q^{l_q} }\Delta^{q-1}f(x_1,\cdots,x_q).
\end{split}
\end{equation*}

The corresponding Newton polynomial  is the unique polynomial of degree $l_1+\cdots l_q-1$ such that, for all $0\le j \le q$ and $0\le l\le l_j-1$, 
$$P^{(l)}(x_j)=f^{(l)}(x_j).$$
We have the following estimate 
\begin{lemma}\label{estim} \cite[Lemma 6.2.9.]{Be-Ga}\par
Let $f\in A(\C)$, $\Omega$ an open set of $\C$, $\delta>0$ and $x_1,\cdots,x_k$ in $\Omega_0=\{z\in \Omega :\ d(z,\Omega^c)>\delta\}$.
Then
$$\vert \Delta^{k-1} f(x_1,\ldots,x_k)\vert \le \frac{2^{k-1}}{\delta^{k-1}}\sup_{z\in \Omega}\vert f(z)\vert.$$
\end{lemma}
Let $B>0$ be fixed and $f \in A_{p,B}(\C)$. 

Let $n$ be a fixed integer. Let $\vert z_j\vert\le 2^n$ and $0\le l\le m_j-1$. We consider the divided differences of $f$ with respect to the points $z_1,\cdots, z_j$, each $z_k$, $1\le k\le j-1$ repeated $m_k$ times and $z_j$ repeated $l$ times..
  Denote by $M_{j,l}=m_1+\cdots+m_{j-1}+l$, the divided differences are
$$\phi_{j,l}=\Delta^{M_{j,l}} f(\underbrace{z_1,...,z_1}_{m_1\textrm{ times}},\dots,\underbrace{z_{j-1},...,z_{j-1}}_{m_{j-1}\textrm{ times}},\underbrace{z_j,...,z_j}_{l+1\textrm{ times}}).$$

Using  Lemma \ref{estim} with $\Omega=D(0,2^{n+2})$, $\delta=2^{n+1}$, $k=M_{j,l}+1$, we have

$$\vert \phi_{j,l}\vert \le 2^{-nM_{j,l}}\Vert f\Vert_B \exp(Bp(2^{n+2}))\le 2^{-nM_{j,l}}\Vert f\Vert_B \exp(B'p(2^n)).$$
Thus, 
$$\Vert \rho(f)\Vert_{B'}\le \Vert f\Vert_B$$
and this concludes the proof of Proposition \ref{nec}.

\bigskip
\end{proof}

Before proceeding with the proofs of the main results, we need the following lemmas :

\begin{lemma}\label{number}
Condition (\ref{N(0,R)}) implies that there exist constants $A,B>0$ such that, for all $R>0$, 
$$n(0,R)\le A p(R)+B.$$
\end{lemma}
\begin{proof} Using property (c) of the weight, we have
$$n(0,R)\le 2\int_R^{2R}\frac{n(0,t)}{t}dt\le 2N(0,2R)\le Ap(2R)+B\le A p(R)+B.$$
\end{proof}

\begin{lemma}\label{induction}
Let $W$ be an element of $A(V)$ and $q$ be in $\N^*$. We suppose that for all  $1\le j\le q$, for all  $n\in \N$ such that  $\vert z_q\vert  \le 2^n$ and for all $0\le l\le m_j-1$, we have
\[
\vert \phi_{j,l}\vert 2^{n(l+m_1+...+m_{j-1})}\le A \exp(B p(2^n)),
\]
where $A$ and $B$ are positive constants only depending on $V$ and $W$.

Then, there exist constants $A,B>0$ only depending on $V$ and $W$, such that, for all $n\in \N$ and $\vert z\vert \le 2^n$,
\[
\begin{split}
\sum_{l=0}^{+\infty}\frac{\vert P_q^{(l)}(z)\vert }{l!} & \le A \exp(Bp(2^n))\sum_{j=1}^q 2^{2(m_1+\cdots+m_j)},\\
\sum_{l=0}^{+\infty}\frac {\vert \Pi_q^{(l)}(z)\vert}{l!} & \le 2^{(n+2)(m_1+\cdots+m_q)}.
\end{split}
\]
\end{lemma}

\begin{proof}
If $\vert z\vert\le 2^{n+1}$, then for $j=1,\cdots,q$, $\vert z-z_j\vert\le 2^{n+2}$,
\[ 
\begin{split}
\vert P_q(z)\vert & \le \sum_{j=1}^q 2^{(n+2)(m_1+\cdots m_{j-1})}\sum_{l=0}^{m_j-1}\vert \phi_{j,l}\vert 2^{(n+2)l}\\
& \le A \exp(Bp(2^n))\sum_{j=1}^q 2^{2(m_1+\cdots+m_j)}
\end{split}
\]
and 
\[
\vert \Pi_q(z)\vert = \prod_{j=1}^q \vert z-z_j\vert^{m_j}\le 2^{(n+2)(m_1+\cdots+m_q)}.
\]

Now for $\vert z\vert \le 2^n$, if  $\vert z-w\vert \le 2$, then $\vert w\vert \le 2^{n+1}$. By the preceding inequalities and Cauchy inequalities, for all $l\ge 0$,
$$\frac{\vert P_q^{(l)}(z)\vert}{l!}\le \frac{1}{2^l} \max_{\vert z-w\vert \le 2} \vert P_q(w)\vert \le \frac {1}{2^l}A \exp(Bp(2^n)) \sum_{j=1}^q 2^{2(m_1+\cdots+m_j)}.$$
We readily obtain the desired estimate for $P_q$. Using Cauchy estimates once again for the function $\Pi_q$ we obtain the second inequality.
\end{proof}

\begin{proof}[Proof of Proposition \ref{subspace}]\par
We assume that condition (\ref{N(0,R)}) holds. Let $W=\{w_{j,l}\}_{j,0\le l\le m_j-1}\in \tilde A_p(V)$.
Let $q\ge 1$ and  $n$ be the integer such that $2^{n-1}\le \vert z_q\vert < 2^n$.
We know that $\displaystyle \frac{P_q^{(l)}(z_q)}{l!}=w_{q,l}$ for every $0\le l\le m_{q-1}$.
By the preceding lemma, 
$$\sum_{l=0}^{m_q-1} \vert w_{q,l}\vert \le \sum_{l=0}^{+\infty}\frac{\vert P_q^{(l)}(z_q)\vert}{l!}  \le A \exp(Bp(2^n))\sum_{j=1}^q2^{2(m_1+\cdots+m_j)}.$$
By Lemma \ref{number}, 
$m_1+\cdots m_j\le n(0,\vert z_j\vert)\le A p(z_j)+B$. Using that $q\le n(0,\vert z_q\vert)\le A p(z_q)+B$, we obtain
$$\sum_{l=0}^{m_q-1} \vert w_{q,l}\vert \le A \exp(Bp(2^n))\le A \exp(Bp(z_q)),$$
that is $W\in A_p(V)$.
\end{proof}

\begin{proof}[Proof of Theorem \ref{intvar}]\par

 We assume that conditions (\ref{N(0,R)}) and (\ref{N(z,z)}) are fulfilled. 
 We already have $\tilde A_p(V)\subset A_p(V)$ by Proposition \ref{subspace}.

Before proving the reverse inclusion, we need  some useful consequences of (\ref{N(0,R)}) and (\ref{N(z,z)}) : 
 
 \begin{lemma}\label{useful}
 There exist constants $A,B>0$ such that, for all  $j\in \N^*$ and for all $n\in  \N$ such that $\vert z_j\vert \le 2^n$, we have 
\begin{itemize}
\item [(i)] $2^{nm_j}\le A\,\vert z_j\vert^{m_j} \exp(Bp(2^n))$,\ \ $2^{n(m_1+\cdots+m_j)}\le A\vert z_j\vert^{m_1+\cdots+m_j} \exp(Bp(2^n))$.
\item [(ii)]  $\vert z_j\vert^{m_j}\le A \exp(Bp(z_j))$,
\item[(iii)] $\prod_{k=1}^{j-1}\vert z_j-z_k\vert^{-m_k} \le A \exp(Bp(2^n)) 2^{-n(m_1+\cdots+m_{j-1})}$.
\end{itemize}
\end{lemma}

\begin{proof}
(i) For $0<\vert z_j\vert \le 2^n$, we have 
$$N(0,2^n)\ge \sum_{0<\vert z_k\vert\le 2^n}m_k\ln\frac{2^n}{ \vert z_k\vert}\ge m_j\ln\frac{2^n}{\vert z_j\vert}.$$
We readily obtain the result by condition (\ref{N(0,R)}). 

The second inequality is obtained in the same way, noting that 
$$N(0,2^n)\ge \sum_{k=1}^jm_k\ln\frac{2^n}{\vert z_k\vert}\ge(\ln \frac{2^n}{\vert z_j\vert}) \sum_{k=1}^j m_k .$$

(ii) It is a simple consequence of condition (\ref{N(z,z)}) :
$$m_j \ln \vert z_j\vert \le N(z_j,\vert z_j\vert)\le A p(z_j)+B.$$

(iii) It is also a consequence of condition (\ref{N(z,z)}) :

$$\sum_{k=1}^{j-1} m_k\ln\frac{ \vert z_j\vert }{\vert z_j-z_k\vert} \le \sum_{0<\vert z_k-z_j\vert\le \vert z_j\vert} m_k\ln\frac{\vert z_j\vert }{\vert z_j-z_k\vert}=N(z_j,\vert z_j\vert)\le A p(z_j)+B.$$ 
We deduce that 

\[
\begin{split}
\prod_{k=1}^{j-1}\vert z_j-z_k\vert^{-m_k} & \le A\exp(Bp(z_j))\vert z_j\vert^{-(m_1+\cdots+m_{j-1})}\\
 & \le A\,2^{-n(m_1+\cdots m_{j-1})}\exp(Bp(2^n))
 \end{split}
 \]
 using (i).
\end{proof}

Let $W=\{w_{j,l}\}_{j,0\le l\le m_j-1}$ be in $A_p(V)$. In order to show that $W$ verifies (\ref{Im}), we are going to use Lemma \ref{induction} and show by induction on $q\ge 1$ the following property :

For all $n\in \N$ such that $\vert z_q\vert  \le 2^n$ and for all $0\le l\le m_q-1$, 
$$\vert \phi_{q,l}\vert 2^{n(l+m_1+...+m_{q-1})}\le A \exp(B p(2^n)),$$
where $A$ and $B$ are positive constants only depending on $V$ and $W$.

$q=1$ : for $\vert z_1\vert \le 2^n$ and  $0\le l\le m_1-1$, we have 
\[
\vert \phi_{1,l}\vert= \vert w_{1,l}\vert \le A \exp(Bp(z_1))\le A \exp(Bp(z_1)) 2^{-nl} 2^{nm_1}\\
 \le  A\exp(Bp(2^n)) 2^{-nl}
\]
using Lemma \ref{useful}, (i) and (ii).

Suppose the property  true for $1\le j\le q-1$. Let $n\in \N$ be such that $\vert z_q\vert \le 2^n$.

Again, we proceed by induction on $l$, $0\le l\le m_q-1$.

$l=0$ : by Lemmas \ref{induction} and \ref{number}, we have 
\[
\vert P_{q-1}(z_q)\vert \le A \exp(Bp(2^n))\sum_{j=1}^{q-1} 2^{2(m_1+\cdots+m_j)}\le (q-1) 2^{2(m_1+\cdots+m_{q-1})}\le A  \exp(Bp(2^n)).
\]
By Lemma \ref{useful} (iii),
\[
\vert \Pi_{q-1}(z_q)\vert^{-1}=\prod_{k=1}^{q-1}\vert z_q-z_k\vert^{-m_k} \le A \exp(Bp(2^n)) 2^{-n(m_1+\cdots+m_{q-1})}
\]
We deduce that 
\[ \vert \phi_{q,0}\vert \le A \exp(Bp(2^n)) 2^{-n(m_1+\cdots+m_{q-1})}.\]

Suppose the estimate true for $0\le j\le l-1$, using both inequalities of Lemma \ref{induction} and Lemma \ref{number}, we have
$$ \sum_{j=0}^{l-1}\vert \frac{\Pi_{q-1}^{(l-j)}(z_q)}{(l-j)!}\phi_{q,j}\vert \le A\exp(Bp(2^n))$$
and 
$$\vert \frac{P_{q-1}^{(l)}(z_q)}{l!}\vert \le A\exp(Bp(2^n)). $$
As for $l=0$, we use Lemma \ref{useful}  (iii) to complete the proof.
\end{proof}

\begin{proof}[Proof of Theorem \ref {main}]\par

We already showed the necessity in Theorem \ref{nec}. Let us prove the sufficiency :

We assume condition (\ref{N(0,R)}). Let $W=\{w_{j,l}\}_{j,0\le l\le m_j-1}$ be an element of $\tilde A_p(V)$.

Let $\mathcal X$ be a smooth cut-off function such that $\mathcal X(x)=1$ if $|x|\le1$ and $\mathcal X(x)=0$ if $|x|\ge4$.

Set  $\mathcal X_n(z)=\mathcal X(\frac{|z|^2}{2^{2n}})$, for $n \in \N$, $\rho_0=\mathcal X_0$ and $\rho_{n+1}=\mathcal X_{n+1}-\mathcal X_n$.
It is clear that the family $\{\rho_n\}_n$ form a partition of the unity, that the support of $\mathcal X_n$ is contained in the disk $|z|\leq2^{n+1}$ and that the support of 
$\rho_n$ is contained in the annulus $\{2^{n-1}\le|z|\leq 2^{n+1}\}$ for $n\ge1$. 

We will denote by $q_n$  the number of distinct points $z_j$ in $D(0,2^n)$, that is : $q_n=\sum_{\vert z_j\vert \le 2^n} 1$.

\begin{lemma}\label {F}
There exists a ${\mathcal C}^{\infty}$ function $F$ on $\C$ such that, for certain constants $A,B>0$,
\begin{itemize}
\item[(i)] $\displaystyle \frac {F^{(l)}(z_j)}{ l!}=w_{j,l}$\  for all $j\in \N$, $0\le l\le m_j-1$.
\item [(ii)] for all $z\in \C$, $\vert F(z)\vert \le Ae^{Bp(z)}$, 
\item [(iii)] $\bar\partial F=0$ on $D(0,1)$ and for  any $n\ge 2$ and $2^{n-2}\le \vert z\vert < 2^{n-1}$,
$$\vert \bar\partial F(z)\vert \le A\  2^{-n(m_1+\cdots+m_{q_n})}\prod_{k=1}^{q_n}\vert z-z_k\vert^{m_k} e^{Bp(2^n)}.$$
\end{itemize}
\end{lemma}

\begin{proof}
We set 
$$F(z)=\sum_{n\ge 2} \rho_{n-2}(z)P_{q_n}(z).$$
where 
$$P_q(z)=\sum_{j=1}^q\left(\sum_{l=0}^{m_j-1} \phi_{j,l}(z-z_j)^l\right)\prod_{k=1}^{j-1}(z-z_k)^{m_k}.$$
It is the Newton polynomial we mentioned in Remark \ref{Newton}.

(i) : For all $j\ge 1$ and $0\le l\le m_j-1$, if  $z_j$ is in the support of $\rho_{n-2}$, then $P_{q_n}^{(l)}(z_j)=l!w_{j,l}$. Thus
\[
\begin{split}
F^{(l)}(z_j) & =\sum_{n\ge 2}\left(\sum_{k=0}^lC_l^k\rho_{n-2}^{(l-k)}(z_j)k!w_{j,k}\right)\\
& = \sum_{k=0}^lC_l^k k!w_{j,k}(\sum_n\rho_n)^{(l-k)}(z_k)=l!\,w_{j,l}.
\end{split}
\]

(ii) : For $z\ge 1$, let $n\ge 2$ be the integer such that $2^{n-2}\le \vert z\vert < 2^{n-1}$. Then, we have  : 
$$F(z)=\rho_{n-2}(z)P_{q_n}(z)+\rho_{n-1}(z)P_{q_{n+1}}(z).$$

For all $0\le j\le q_n$, we have $\vert z_j\vert \le 2^n$ and $|z-z_j|\le2^{n+1}$.
Using Lemmas \ref {induction}, condition (\ref{N(0,R)})  and property (c) of the weight, we have
$$\vert P_{q_n}(z)\vert \lesssim  \exp(Bp(2^n))\le A \exp(B p(2^n))\le A \exp(Bp(z)).$$

The same estimation holds for $P_{q_{n+1}}$ thus,
$$\vert F(z)\vert \lesssim \exp(Bp(z)).$$

(iii) Now, we want to estimate $\bar\partial F$. 

It is clear that $F(z)=P_{q_2}(z)$  on $D(0,1)$.  

Let $\vert z\vert \ge 1$ and $n$ the integer such that $2^{n-2}\le |z| <2^{n-1}$. We have

$$\bar\partial F(z)=\bar\partial \rho_{n-2}(z) P_{q_n}(z)+\bar\partial \rho_{n-1}(z)P_{q_{n+1}}(z).$$

Since $z$ is outside the supports of  $\bar \partial  \mathcal X_{n-3}$ and of $\bar \partial  \mathcal X_{n-1}$, we have
$$\bar\partial F(z)=-\bar \partial \mathcal X_{n-2}(z) (P_{q_{n+1}}(z)-P_{q_n}(z))=\prod_{k=1}^{q_n}(z-z_k)^{m_k} G_n(z)$$
where 
$$G_n(z)=-\bar \partial \mathcal X_{n-2}(z)\sum_{j=q_n+1}^{q_{n+1}}\prod_{k=q_n+1}^{j-1}(z-z_k)^{m_k}\left(\sum_{l=0}^{m_j-1} \phi_{j,l}(z-z_j)^l\right).$$
For $k\le q_{n+1}$, $\vert z-z_k\vert \le 2^{n+2}$, thus, using the estimate given by (\ref{Im}) then Lemma \ref{number}, we show that

\begin{equation*}
\begin{split}
\vert G_n(z)\vert & A \exp(Bp(2^n)) 2^{-n(m_1+\cdots+q_n)} \sum_{j=q_n+1}^{q_{n+1}} 2^{m_{q_n+1}+...+m_j}\\
& \lesssim \exp(Bp(2^n))2^{-n(m_1+\cdots+m_{q_n})}.
\end{split}
\end{equation*}
We readily obtain the desired estimate.
\end{proof}

Now, when looking for a holomorphic interpolating function of the form
$f=F+u$, we are led to the $\bar\partial$-problem
\[
\bar\partial u=-\bar\partial F\ , 
\]
which we solve using H\"ormander's theorem \cite[Theorem 4.2.1]{Ho1}.

The interpolation problem is then reduced to the following lemma.
\begin{lemma}\label{U}

There exists a subharmonic function $U$ such that, for certain constants $A,B>0$,
\begin{itemize}
\item[(i)] $U(z)\simeq m_j\log |z-z_j|^2$ near $z_j$,
\item[(ii)] $U(z)\le A p(z) +B$ for all  $z\in \C $.
\item[(iii)] $\vert \bar\partial F(z)\vert^2 e^{-U(z)}\le A e^{B( p(z)}$ for all $z\in\C$.
\end{itemize}
\end{lemma}

Admitting this lemma for a moment, we proceed with the proof of the theorem. 

 From H\"ormander theorem \cite[Theorem 4.4.2]{Ho1}, we can find a ${\mathcal{C}}^\infty$ function $u$ such that 
$\bar\partial u=-\bar\partial F$ and, denoting by $d\lambda$ the Lebesgue measure,
$$\int_{\C}\frac{\vert u(w) \vert^2e^{-U(w)-Ap(w)}}{(1+\vert w\vert^2)^2}\,d\lambda(w)\le\int_{\C}\vert \bar\partial F \vert^2e^{-U(w)-Ap(w)}\,d\lambda(w).$$
By the property (a) of the weight $p$, there exists $C>0$ such that
$$\int_{\C}e^{-Cp(w)}d\lambda(w)<\infty.$$
Thus, using (ii) of the lemma, and the estimate on $|\bar\partial F(z)|^2$, we see that  the last integral is convergent if  $A$ is large enough.
By condition (iii), near $z_j$, $e^{-U(w)}(w-z_j)^l$ is not summable for $0\le l\le m_j-1$, so we have necessarily $u^{(l)} (z_j)=0$ for all $j$ and $0\le l \le m_j-1$ and consequently, 
$\displaystyle \frac{f^{(l)}(z_j)}{l!}=w_j^l$.

Now, we have to verify that $f$ has the desired growth. 

By the mean value inequality,
$$
\vert f(z)\vert\lesssim \int_{D(z,1)}\vert f(w)\vert\,d\lambda(w)\lesssim
\int_{D(z,1)}\vert F(w)\vert\,d\lambda(w)
+\int_{D(z,1)}\vert u(w)\vert\,d\lambda(w).
$$
Let us estimate the two integrals that we denote by $I_1$ and $I_2$.

\noindent For $w\in D(z,1)$, 
$$\vert F(w)\vert\lesssim e^{Bp( w)}\lesssim e^{Cp( z)}.$$
Then, 
$$I_1\lesssim e^{C
p(z)}$$
\smallskip
\noindent To estimate $I_2$, we use Cauchy-Schwarz inequality,
$$I_2^2\le J_1\,J_2$$
where
$$
J_1=\int_{D(z,1)}\vert u(w)\vert^2e^{-U(w)-Bp( w)}\,d\lambda(w),\ J_2=\int_{D(z,1)}e^{U(w)+Bp(w)}\,d\lambda(w).$$
We have 
$$J_1 \lesssim \int_{\C} \vert u(w) \vert^2e^{-U(w)-Bp(w)}\,d\lambda(w) \lesssim \int_{\C}\frac{\vert u(w) \vert^2e^{-U(w)}}{(1+\vert w\vert^2)^2}\,d\lambda(w)< +\infty,$$
by property (a) of $p$, if $B>0$ is chosen big enough.

To estimate $J_2$, we use the condition (i) of the lemma and the property (b) of the weight $p$. For $w\in D(z,1)$,  
$$e^{U(w)+Bp(w)}\le e^{Cp( w)}\lesssim e^{Ap(z)}.$$
We easily deduce that $J_2\lesssim e^{Ap(z)}$ and, finally, that $f\in A_p(\C)$.
\end{proof}

\begin{proof}[Proof of Lemma \ref{U}]

For the sake of simplicity and up to a homotethy, we may assume that $\vert z_k\vert > 2$ for all $z_k\not=0$. Besides, in the definition of the following functions $V_n$, we will assume $z_1\not=0$, otherwise, we may add the term $m_1\ln \vert z\vert$ to each $V_n$.
We set 
$$V_n(z)=\sum_{ 0<|z_j|\le  2^n}m_j\log\frac{ |z-z_j|^2}{|z_j|^2},$$
then 
$$V(z)=\sum_{n\ge 2} \rho_{n-2}(z)V_n(z).$$

First, we will show that $V$ verifies (i), (ii) and (iii).  Then, we will estimate $\Delta V$ from below and add a correcting term $W$. The subharmonic function $U$ will be of the form $V+ W$.

(i) Let $\vert z_k\vert$ be such that $2^{m-1}< \vert z_k\vert<2^{m+1}$. For $2^{m-1}< \vert z\vert<2^{m+1}$,
$$V(z)=\rho_{m-1}(z)V_{m+1}(z)+\rho_m(z)V_{m+2}(z)+\rho_{m+1}(z)V_{m+3}(z).$$
As the $\rho_n$'s form a partition of the unity, it is clear that $V(z)- m_k\ln\vert z-z_k\vert^2$ is continuous in a neighborhood of $z_k$. 

Note that $V$ is smooth on $\{\vert z\vert \le 2\}$ since we have assumed that all $\vert z_j\vert > 2$.

(ii) Let $n\ge 2$ and $2^{n-2} \le |z|<2^{n-1}$.
then 
$$V(z)=\rho_{n-2}(z)V_n(z)+\rho_{n-1}(z)V_{n+1}(z).$$
For all $|z_j|<2^n$, we have  $|z-z_j|< 2^{n+1}$.
Thus, 
$$V_n(z)\le\sum_{ |z_j|\le  2^n}m_j\log \frac{2^{n+1}}{|z_j|}\le N(0,2^{n+1}).$$
 Finally, we obtain that 
 $$V(z)\le N(0,2^{n+1})+N(0,2^{n+2})\lesssim p(2^n) \lesssim p(z)$$
  by condition (\ref{N(0,R)}) and property (c) of the weight.

(iii)  We  have
$$-V(z)/2=\sum_{\vert z_j\vert\le 2^n} m_j \ln \frac{\vert z_j\vert}{\vert z-z_j\vert}+\rho_{n+1}(z)\sum_{2^n<\vert z_j\vert\le 2^{n+1}}m_j\ln\frac{\vert z_j\vert}{\vert z-z_j\vert}.$$
Note that for all $2^n<\vert z_j\vert\le 2^{n+1}$, we have $\vert z-z_j\vert > 2^n-2^{n-1}=2^{n-1}$. We obtain
\begin{equation}
\begin{split}
-V(z)/2 \le 
& \sum_{\vert z_j\vert\le 2^n}  m_j \ln \frac{2^n}{\vert z-z_j\vert}+\ln 4  \sum_{2^n<\vert z_j\vert\le 2^{n+1}}m_j\\
& \le \ln \left(2^{n(m_1+\cdots m_{q_n})} \prod_{j=1}^{q_n}  \vert z-z_j\vert^{-m_j}\right)+\ln (A \exp(Bp(2^n))
\end{split}
\end{equation}
for certain constants $A,B>0$ using Lemma \ref{number}.
Finally, combining this inequality with (iii) of Lemma \ref{F}, we obtain 
$$\vert \bar\partial F(z)\vert \exp(-V(z)/2)\lesssim \exp(Bp(2^n))\lesssim \exp(Bp(z)).$$

Now, in order to get a lower bound of  the laplacian, we compute $\Delta V(z)$ :
$$\Delta V=\sum_{n\ge 2}\rho_{n-2} \Delta V_n+2\Re \left(\sum_n\bar\partial \rho_{n-2} \partial V_n\right)+\sum_{n\ge 2} \partial\bar \partial\rho_{n-2} V_n.$$
The first sum is positive since every $V_k$ is subharmonic.

Let us estimate  the second and the third sums, that we will denote respectively by 
$B(z)$ and $C(z)$. 
For $n\ge 2$ and $2^{n-2} \le |z|<2^{n-1}$, since $z$ is outside the supports of  $\bar \partial  \mathcal X_{n-3}$ and of $\bar \partial  \mathcal X_{n-1}$, we have

\[
\begin{split}
B(z)= & 2\Re \left [\bar \partial \mathcal X_{n-2}(z)\partial \left(V_n(z)-V_{n+1}(z)\right)\right],\\
C(z)= & \partial \bar \partial \mathcal X_{n-2}(z) \left(V_n(z)-V_{n+1}(z) \right).
\end{split}
\]

 $$V_n(z)-V_{n+1}(z)=\sum_{ 2^n< |z_j|\le 2^{n+1}}m_j\log\frac{|z-z_j|^2}{|z_j|^2},$$
 $$\partial \left(V_n(z)-V_{n+1}(z)\right)=\sum_{ 2^n<|z_j|\le 2^{n+1}}m_j\frac{1}{z- z_j},$$
 and
 $$\vert \bar\partial \mathcal X_{n-2}(z)\vert \lesssim \frac {1}{2^n},\qquad |\partial \bar\partial \mathcal X_{n-2}(z)|\lesssim \frac{1}{2^{2n}}.$$
 For  $z$  in the support of $\bar \partial \mathcal X_{n-2}$, we have $|z|\le2^{n-1}$, and for $2^n\le|z_j|<2^{n+1}$, $2^{n-1}\le|z-z_j|\le2^{n+2}$. Thus, we obtain that 
$$| \partial \bar \partial \mathcal X_{n-2}(z) \left(V_{n+1}(z)-V_n(z)\right)| \lesssim \frac{n(0,2^{n+1})-n(0,2^n)}{2^{2n}},$$
and 
 $$|\bar \partial \mathcal X_{n-2}(z)\partial \left(V_{n+1}(z)-V_n(z)\right)| \lesssim \frac{n(0,2^{n+1})-n(0,2^n)}{2^{2n}}.$$
 Finally, 
 $$\Delta V(z)\gtrsim -\frac{n(0,2^{n+1})-n(0,2^n)}{ 2^{2n}}\gtrsim - \frac{n(0,2^3|z|)-n(0,2|z|)}{ |z|^2} .$$
To construct the correcting term, $W$, we begin by putting
$$ f(t)=\int_0^t n(0,s)ds,\  \ g(t)=\int_0^t \frac{f(s)}{s^2} ds \ \hbox{and  }\ \ W(z)=g(2^3 |z|).$$
The following inequalities are easy to see :
$$f(t)\le t n(0,t),\ \ g(t)\le \int_0^t \frac{n(0,s)}{s} ds=N(0,s).$$
Thus,  by condition (\ref {N(0,R)}) and property (c), 
$$W(z)\le N(0,2^3|z|)\lesssim p(2^3 z)\lesssim p(z)$$
Finally,  to estimate the laplacian of $W$, we will denote $t=2^3 |z|$.
$$\Delta W(z)=\frac{1}{t} g'(t)+g''(t)=\frac{1}{t^2} (f'(t)-\frac{f(t)}{t}).$$
$$f(t)=\int_0^t n(0,s)ds = \int_0^\frac{t}{4} n(0,s)ds+\int_\frac{t}{4}^t n(0,s)ds\le\frac{t}{4}n(0,\frac{t}{4})+t(1- \frac{1}{ 4}) n(0,t).$$
Thus, $$f'(t)-\frac{f(t)}{t}=n(0,t)-\frac{f(t)}{t} \ge\frac{1}{4}(n(0,t)-n(0,\frac{t}{4}))$$
and $$\Delta W(z)\gtrsim \frac{n(0,2^3|z|)-n(0,2|z|)}{|z|^2}.$$
Now, the desired function will be of the form 
$$U(z)=V(z)+\alpha W(z),$$
where $\alpha$ is a positive constants chosen big enough.
\end{proof}

\end{document}